%% file: chromthreshzero-ams.tex
\documentclass[12pt]{article}
\usepackage{fullpage,amsthm,amsmath,amssymb,tikz}

\newcounter{thmctr}
\newtheorem{thm}[thmctr]{Theorem}
\newtheorem{lemma}[thmctr]{Lemma}

\newtheorem{cor}[thmctr]{Corollary}
\newtheorem{conj}[thmctr]{Conjecture}
\newtheorem{problem}[thmctr]{Problem}
\theoremstyle{definition}
\newtheorem*{defi}{Definition}
\theoremstyle{plain}
\DeclareMathOperator*{\diam}{\text{diam}}

\newcommand{\jozifoot}{jobal@math.uiuc.edu; Department of Mathematics, University of Illinois, 1409
W. Green Street, Urbana, IL 61801; and Mathematical Institute, University of Szeged, Szeged,
Hungary; Research supported in part by:  Marie Curie Fellowship IIF-327763, NSF CAREER Grant
DMS-0745185, UIUC Campus Research Board Grants 11067 and 13039 (Arnold O.\ Beckman Research Award),
and OTKA Grant K76099.}

\newcommand{\johnfoot}{lenz@math.uic.edu; Department of Mathematics, Statisticts, and Computer
Science, University of Illinois at Chicago, 851 S. Morgan Street, Chicago, IL; Research partly
supported by NSA Grant H98230-13-1-0224.}

\title{Hypergraphs with Zero Chromatic Threshold}
\author{J\'ozsef Balogh\footnote{\jozifoot} \and John Lenz\footnote{\johnfoot}}

\begin{document}
\maketitle
\input{abstract.tex}
\input{ctz-intro.tex}
\input{ctz-lower.tex}
\input{ctz-upper.tex}
\bibliographystyle{abbrv}
\bibliography{refs}
\end{document}

%% file: abstract.tex
\begin{abstract}
  Let $F$ be an $r$-uniform hypergraph.  The \emph{chromatic threshold} of the family of $F$-free,
  $r$-uniform hypergraphs is the infimum of all non-negative reals $c$ such that the subfamily of
  $F$-free, $r$-uniform hypergraphs $H$ with minimum degree at least $c \binom{|V(H)|}{r-1}$ has
  bounded chromatic number.  The study of chromatic thresholds of various graphs has a long history,
  beginning with the early work of Erd\H{o}s-Simonovits.  One interesting question, first proposed
  by \L{}uczak-Thomass\'{e} and then solved by Allen-B\"{o}ttcher-Griffiths-Kohayakawa-Morris, is
  the characterization of graphs having zero chromatic threshold, in particular the fact that there
  are graphs with non-zero Tur\'{a}n density that have zero chromatic threshold.  In this paper, we
  make progress on this problem for $r$-uniform hypergraphs, showing that a large class of
  hypergraphs have zero chromatic threshold in addition to exhibiting a family of constructions
  showing another large class of hypergraphs have non-zero chromatic threshold.  Our construction is
  based on a special product of the Bollob\'as-Erd\H{o}s graph defined earlier by the authors.
\end{abstract}

%% file: ctz-intro.tex
\section{Introduction} 
\label{sec:Introduction}

In 1973, Erd\H{o}s and Simonovits~\cite{cr-erdos73} asked the following question: ``If $G$ is
non-bipartite, what bound on $\delta(G)$ forces $G$ to contain a triangle?''  This question was
answered by Andr\'{a}sfai, Erd\H{o}s, and S\'{o}s~\cite{cr-andrasfai74}, where they showed that if
$G$ is a triangle-free, $n$-vertex graph with $\delta(G) > \frac{2n}{5}$, then $G$ is bipartite.
The blowup of $C_5$ shows that this is sharp.  Erd\H{o}s and Simonovits~\cite{cr-erdos73}
generalized this problem to the following conjecture: if $\delta(G) > (1/3 + \epsilon) \left| V(G)
\right|$ and $G$ is triangle-free, then $\chi(G) < k_{\epsilon}$, where $k_{\epsilon}$ is a constant
depending only on $\epsilon$.  The conjecture was proven by Thomassen~\cite{cr-thomassen02}, but
interest remained in generalizing the problem to other graphs and to hypergraphs.

\begin{defi}
  An \emph{$r$-uniform hypergraph} $H$ is a pair $(V(H),E(H))$ where $V(H)$ is any finite set and
  $E(H)$ is a family of $r$-subsets of $V(H)$.  A set of vertices $X \subseteq V(H)$ is an
  \emph{independent set} if every hyperedge (element of $E(H)$) contains at least one vertex outside
  $X$ and $X$ is a \emph{strong independent set} if every hyperedge intersects $X$ in at most one
  vertex.  The \emph{degree} of a vertex $x \in V(H)$, denoted $d(x)$, is the number of hyperedges
  containing $x$.  The \emph{minimum degree of $H$}, denoted $\delta(H)$, is the minimum degree of a
  vertex of $H$.  A hypergraph $H$ is \emph{$k$-colorable} if there exists a partition of the vertex
  set of $H$ into $k$ sets $V_1,\dots,V_k$ such that each $V_i$ is an independent set. The
  \emph{chromatic number of $H$}, denoted $\chi(H)$, is the minimum $k$ such that $H$ is
  $k$-colorable.  If $F$ and $H$ are $r$-uniform hypergraphs, then $H$ is \emph{$F$-free} if $F$
  does not appear as a subhypergraph of $H$, i.e.\ there does not exist an injection $\alpha : V(F)
  \rightarrow V(H)$ such that if $\{f_1,\dots,f_r\}$ is a hyperedge of $F$ then
  $\{\alpha(f_1),\dots,\alpha(f_r)\}$ is a hyperedge of $H$. If $X \subseteq V(H)$, the
  \emph{induced hypergraph on $X$}, denoted $H[X]$, is the hypergraph with vertex set $X$ and
  its hyperedges are all hyperedges of $H$ which are completely contained inside $X$.
\end{defi}

\begin{defi}
  Let $\mathcal{F}$ be a family of $r$-uniform hypergraphs.  The \emph{chromatic threshold} of
  $\mathcal{F}$ is the infimum of $c \geq 0$ such that the subfamily of $\mathcal{F}$
  consisting of hypergraphs $H$ with minimum degree at least $c \binom{\left| V(H) \right|}{r-1}$
  has bounded chromatic number.
\end{defi}

Note that Erd\H{o}s and Simonovits'~\cite{cr-erdos73} conjecture is that the chromatic threshold of
the family of triangle-free graphs is $1/3$.  The chromatic threshold of the family of $F$-free
graphs for various $F$ have been studied by several researchers~\cite{cr-allen11, cr-brandt11,
cr-goddard10, cr-luczak06, cr-luczak10, cr-thomassen02, cr-thomassen07}, eventually culminating in a
theorem of \L{}uczak and Thomass\'{e}~\cite{cr-luczak10} and a theorem of Allen, B\"{o}ttcher,
Griffiths, Kohayakawa, and Morris~\cite{cr-allen11} where they determined the chromatic threshold of
the family of $F$-free graphs for all $F$.  An interesting consequence of \cite{cr-allen11} is the
solution to a conjecture of \L{}uczak and Thomass\'{e}~\cite{cr-luczak10}: the family of $F$-free
graphs has chromatic threshold zero if and only if $F$ is near acyclic.  (A graph
$G$ is \emph{near acyclic} if there exists an independent set $S$ in $G$ such that $G-S$ is a forest
and every odd cycle has at least two vertices in $S$.)  This is surprising because the Tur\'{a}n
density is a trivial upper bound on the chromatic threshold, but the family of graphs with zero
Tur\'{a}n density (bipartite graphs) differs from the family of near-acyclic graphs.

Balogh, Butterfield, Hu, Lenz, and Mubayi~\cite{cr-balogh11} initiated the study of chromatic
thresholds of $r$-uniform hypergraphs and among other things proposed the following problem, again
interested in comparing zero Tur\'{a}n density with zero chromatic threshold.

\begin{problem} \label{prob:chromzero}
  Characterize the $r$-uniform hypergraphs $F$ for which the chromatic threshold of the family of
  $F$-free hypergraphs has chromatic threshold zero.
\end{problem}

Balogh, Butterfield, Hu, Lenz, and Mubayi~\cite{cr-balogh11} made partial progress on this problem
by proving that for a large class of hypergraphs $F$, the family of $F$-free hypergraphs has
chromatic threshold zero.  In the other direction, \cite{cr-balogh11} also contains constructions of
families of $F$-free hypergraphs with non-zero chromatic threshold for various hypergraphs $F$.  Our
main result in this paper is to extend both of these results, enlarging the class of $F$s for which
we can prove the family of $F$-free hypergraphs has chromatic threshold zero in addition to giving a
more general construction of families with non-zero chromatic threshold.  One of our key ideas is to
use a hypergraph extension of the Bollob\'{a}s-Erd\H{o}s graph~\cite{beg-bollobas76} defined by the
authors in~\cite{rt-balogh11-2,rt-balogh11} instead of the Borsuk-Ulam graph in a Hajnal type
construction used by \L{}uczak and Thomass\'{e}~\cite{cr-luczak10}. To state our results, we need
some definitions.

\begin{defi}
  A \emph{cycle of length $t \geq 2$} in a hypergraph is a collection of $t$ distinct vertices $X =
  \left\{ x_1,\ldots,x_t \right\}$ of $H$ and $t$ distinct edges $E_1,\ldots,E_t$ such that $\left\{
  x_i,x_{i+1} \right\} \in E_i$ for each $i = 1,\ldots, t$ (indices taken mod $t$.)  A hypergraph is
  a \emph{hyperforest} if it contains no cycles.  A hypergraph is a \emph{hypertree} if it is a
  connected hyperforest, where \emph{connected} means for every two vertices $x,y$, there is a
  sequence of edges $E_1,\ldots,E_{\ell}$ for some $\ell$ such that $x \in E_1$, $y \in E_{\ell}$
  and $E_i \cap E_{i+1} \neq \emptyset$ for all $i = 1,2,\ldots,\ell-1$.  A \emph{component} of $H$
  is a maximal connected subhypergraph of $H$.  A hypergraph is \emph{linear} if every pair of
  hyperedges intersect in at most one vertex.
\end{defi}  

The key definitions are the following.

\begin{defi}
  An $r$-uniform hypergraph $H$ is \emph{unifoliate $r$-partite} if there exists a partition of the
  vertices into $r$ classes $V_1, V_2, \ldots, V_{r}$ such that $H[V_1]$ is a linear hyperforest,
  every edge not in $H[V_1]$ has exactly one vertex in each $V_i$, and every cycle in $H$ uses
  either vertices from at least two components of $H[V_1]$ or it contains zero or at least two edges
  of $H[V_1]$.

  An $r$-uniform hypergraph $H$ is \emph{strong unifoliate $r$-partite} if it is \emph{unifoliate
  $r$-partite} and in addition, in a witnessing partition there does not exist two vertices $x,y \in
  V_1$ such that $x$ and $y$ are in the same component of $H[V_1]$ and $H$ contains a sequence of
  hyperedges $E_1,\dots,E_\ell$ for some $\ell$ with $x \in E_1$, $y \in E_\ell$, and $E_i \cap
  E_{i+1} \cap (V_2 \cup \dots \cup V_r) \neq \emptyset$ for all $1 \leq i \leq \ell - 1$.  Note that
  this sequence of hyperedges is like a path between $x$ and $y$ which is required to ``connect''
  using vertices outside $V_1$.
\end{defi}

Strong and normal unifoliate $r$-partite are similar but not quite identical requirements; they
differ only in the type of cycles that are allowed.  To be strong unifoliate $r$-partite, a
hypergraph is forbidden to have cycles which use arbitrary number of edges inside some hypertree of
$H[V_1]$ combined with cross-edges which connect using vertices outside $V_1$.  To be unifoliate
$r$-partite, a hypergraph is forbidden to have cycles using exactly one edge $E$ from $H[V_1]$
together with cross-edges which connect using vertices either outside $V_1$ or vertices that are in
the same hypertree as the edge $E$.  Note that the forbidden cycle condition for strong unifoliate
$r$-partite is a stronger forbidden cycle condition, since if a cycle uses one edge $E$ from
$H[V_1]$ and some cross-edges which connected using vertices inside the same hypertree as $E$, the
cycle could be rerouted through the hypertree containing $E$ decreasing the number of cross-edges
which connect using a vertex of $V_1$.  Our main theorem is the following.

\begin{thm} \label{mainthm}
  Fix $r \geq 3$. If $F$ is an $r$-uniform, strong unifoliate $r$-partite hypergraph then the family of
  $F$-free hypergraphs has chromatic threshold zero. If $F$ is not unifoliate $r$-partite then the
  family of $F$-free hypergraphs has chromatic threshold at least $(r-1)! r^{-2r+2}$.
\end{thm}

The constant $(r-1)!r^{-2r+2}$ could be slightly improved (see the comments in the proof of
Lemma~\ref{lem:constrmindeg}).  Currently, proving sharpness seems out of reach so we make no effort
to optimize the constant.  Theorem~\ref{mainthm} generalizes results of Balogh, Butterfield, Hu,
Lenz, and Mubayi~\cite{cr-balogh11}, since the classes of hypergraphs studied in \cite{cr-balogh11}
are subclasses of either non-unifoliate or strong unifoliate $r$-partite hypergraphs.  For $r \geq
3$, there exist hypergraphs which are unifoliate $r$-partite but not strong unifoliate $r$-partite
so Theorem~\ref{mainthm} does not completely solve Problem~\ref{prob:chromzero}.  An example for $r
= 3$ is the hypergraph with vertex set $\{a_1,a_2,a_3,a_4,b_1,b_2,c_1,c_2\}$ and hyperedges
$a_1a_2a_3$, $a_1b_1c_1$, $a_2b_2c_2$, $a_4b_1c_2$, $a_4b_2c_1$.  An illustrative vertex partition
has $a_1,a_2,a_3,a_4$ in one part, $b_1, b_2$ in the second part, and $c_1,c_2$ in the third part.

\begin{conj} \label{conj:unifoliate}
  For $r \geq 3$ and an $r$-uniform hypergraph $F$, the family of $F$-free hypergraphs has chromatic
  threshold zero if and only if $F$ is unifoliate $r$-partite.
\end{conj}

The definition of unifoliate $2$-partite is not quite equivalent to the definition of a near-acyclic
graph which is why Theorem~\ref{mainthm} above is stated only for $r \geq 3$, but with a little
extra work the proofs below extend to $r = 2$ with unifoliate $2$-partite replaced by near-acylic.
But since different behavior occurs for $r \geq 3$ compared to $r = 2$ (as evinced by the difference
between unifoliate $2$-partite and near-acyclic), we simplify the presentation by focusing only on
$r \geq 3$.  The remainder of this paper is organized as follows. In Section~\ref{sec:lower}, using
a construction built from the high-dimensional unit sphere, we prove that for every non unifoliate
$r$-partite hypergraph $F$, the family of $F$-free hypergraphs has non-zero chromatic threshold.  In
Section~\ref{sec:upperbound}, we show how the tools from \cite{cr-balogh11} can be applied to prove
that if $F$ is a strong unifoliate $r$-partite hypergraph, then the family of $F$-free hypergraphs
has chromatic theshold zero.

%% file: ctz-lower.tex
\section{Construction for positive chromatic threshold} 
\label{sec:lower}

To prove a lower bound on the chromatic threshold of the family of $F$-free hypergraphs, we need to
construct an infinite sequence of $F$-free hypergraphs with large chromatic number and large minimum
degree.  First, we need a construction of Balogh and Lenz~\cite{rt-balogh11-2,rt-balogh11} of a
hypergraph with large chromatic number built from the high dimensional unit sphere.  The
construction is based on the celebrated Bollob\'{a}s-Erd\H{o}s Graph~\cite{beg-bollobas76}.  We only
sketch its definition here, for details see~\cite{rt-balogh11-2,rt-balogh11}.  Throughout this
section, an integer $r$ is fixed and all hypergraphs considered are $r$-uniform.

\begin{defi} \cite[Section 2 with $s = 2$]{rt-balogh11-2}
  Given integers $n$, $L$, and $d$ and a real number $\theta > 0$, we construct the $r$-uniform
  hypergraph $H(n,L,d,\theta)$ as follows.  Let $P$ be $n$ ``evenly distributed'' points on the
  $d$-dimensional unit sphere $\mathbb{S}^d$, let $\ell = \left\lceil \log_2 r \right\rceil$, and
  let $B_1, \dots, B_\ell$ be complete bipartite graphs on the vertex set $[r]$ defined as follows.
  Assign bit strings of length $ \left\lceil \log_2 r \right\rceil$ to the elements of $[r]$ and
  define $B_i$ as the complete bipartite graph consisting of the edges between vertices differing in
  coordinate $i$.  Note that the union $\cup B_i$ covers the complete graph $K_{[r]}$.  If $\bar{x}
  \in P^{\ell}$, for $1 \leq i \leq \ell$ denote by $x_i$ the point of the sphere appearing in the
  $i$-th coordinate of tuple $\bar{x}$.  Let $\left\lVert \cdot \right\rVert$ be the Euclidean norm
  in $\mathbb{R}^{d+1}$.  Now we define an $r$-uniform hypergraph $H'$ as follows.
  \begin{itemize}
    \item $V(H') = P^\ell$,

    \item $E \in \binom{V(H')}{r}$ is a hyperedge if there exists some ordering $\bar{x}^1, \dots,
    \bar{x}^r$ of the elements of $E$ such that for all $1 \leq i \leq \ell$ and all $ab \in
    E(B_i)$, then $\left\lVert x^a_{i} - x^b_{i} \right\rVert > 2 - \theta$.
  \end{itemize}
  Now form the hypergraph $H$ by applying \cite[Theorem 16]{rt-balogh11} to $H'$; this operation
  consists of blowing up every vertex in $H'$ into a strong independent set, randomly sparsening the
  hypergraph, and finally deleting one edge from each cycle on at most $L$ vertices.
\end{defi}

\begin{thm} \label{thm:Ssmallind} (\cite[Lemma 13]{rt-balogh11-2})
Given any real $\delta > 0$ and any integer $L \geq 3$, it is possible to select $d$, $\theta_0 >
0$, and $n_0$ so that $\alpha(H(n,L,d,\theta)) \leq \delta |V(H(n,L,d,\theta))|$ for all $0 < \theta \leq
\theta_0$ and $n \geq n_0$.  Also, every $L$ vertices in $H(n,L,d,\theta)$ induce a linear
hyperforest.
\end{thm}

We will use a subhypergraph of $H$ for our application.  Define
\begin{align*}
  V'_0 = \left\{ (p_1,\dots,p_{\ell}) \in P^{\ell} : d(p_i, p_j) \leq \sqrt{2} \text{ for all $i \neq
  j$}\right\}.
\end{align*}
Note that $V'_0$ is a subset of the vertices of $H'$.  Define $V_0$ as the set of vertices of $H$
which came from a blowup of a vertex of $V'_0$.  Let $H_0 = H[V_0]$.  An easy consequence of
Theorem~\ref{thm:Ssmallind} and some properties of the unit sphere is the following.

\begin{cor} \label{cor:Slargechrom}
Given any integers $L \geq 3$ and $k \geq 2$, it is possible to select $d$, $\theta_0 > 0$, and $n_0$ so that
$\chi(H_0(n,L,d,\theta)) \geq k$ for all $0 < \theta \leq \theta_0$ and $n \geq n_0$.  Also, every $L$
vertices in $H_0(n,L,d,\theta)$ induce a linear hyperforest.
\end{cor}

\begin{proof} 
First, $|V'_0| \geq 2^{-\ell^2} n^\ell$ for large $n$.  Indeed, for each point $p \in P$, there are
roughly $n/2$ points within distance $\sqrt{2}$ of $p$.  In the worst case, to form a vertex
$(p_1,\dots,p_\ell)$ of $H_0$ there are $n$ ways to choose $p_1$, $n/2$ ways to choose $p_2$, $n/4$
ways to choose $p_3$, and so on.  Thus there are certainly at least $2^{-\ell^2} n^{\ell}$ vertices
of $H'_0$, even taking into account that points are evenly distributed on the sphere and there might
not be exactly $n/2$ points within distance $\sqrt{2}$ of $p$ for every $p$ ($n$ is large).

Since $n^\ell = |V(H')|$ and each vertex in $H'$ is blown up into the same number of vertices,
$|V(H_0)| \geq 2^{-\ell^2} |V(H)|$.  Now select $\delta = k^{-1} 2^{-\ell^2}$ and apply
Theorem~\ref{thm:Ssmallind} to obtain $d$, $\theta_0$, and $n_0$.  Now for $ 0 < \theta \leq
\theta_0$ and $n \geq n_0$ (and $n$ large enough for the previous paragraph to hold),
Theorem~\ref{thm:Ssmallind} implies that $\alpha(H_0) \leq \alpha(H) \leq k^{-1} 2^{-\ell^2}
|V(H)| \leq \frac{1}{k} |V(H_0)|$.  Thus $\chi(H_0) \geq k$ and the proof is complete.
\end{proof} 

We also need a simple property of the unit sphere (see~\cite[Property (P1)]{rt-balogh11}.)

\begin{lemma} \label{lem:sphereprop}
  Let $\mu$ be the Lebesgue measure on the unit sphere normalized so that $\mu(\mathbb{S}^d) = 1$.
  For any $\epsilon > 0$ there is a $\beta > 0$ depending only on $\epsilon$ so that
  for any $d \geq 3$ and any fixed $p \in \mathbb{S}^d$,
  \begin{align*}
    \mu\left(\left\{ q \in \mathbb{S}^d : \left\lVert p - q \right\rVert \leq \sqrt{2} - \beta
    \right\}\right) \geq \frac{1}{2} - \epsilon.
  \end{align*}
\end{lemma}

\begin{defi}
  Given an $r$-uniform hypergraph $F$, integers $k,n \geq 2$, and an $\epsilon > 0$, we define an
  $r$-uniform hypergraph $G = G(F,k,\epsilon,n)$ as follows.  First, pick $\beta$ according to
  Lemma~\ref{lem:sphereprop}, let $L = |V(F)|$, and pick $d$, $\theta_0$, and $n_0$ according to
  Corollary~\ref{cor:Slargechrom}.  For $n \leq n_0$, let $G$ be the edgeless hypergraph.  Now
  assume $n > n_0$.  Let $f = |E(F)|$.  Next, define $\theta > 0$ so that $\theta < \theta_0$, $4^f
  \theta^{2^{-f}} < \frac{1}{10}$, and for any fixed $p \in \mathbb{S}^d$
  \begin{align} \label{eq:defofthetadiam}
    \diam\left( \left\{ q \in \mathbb{S}^d : \left\lVert p - q \right\rVert \leq \sqrt{2} - \beta
    \right\} \right) < 2 - 4^{f}\theta^{2^{-f}}.
  \end{align}
  This is possible since once $\beta > 0$ is chosen, the spherical cap centered at $p$ will not be
  the entire hemisphere so $\theta$ can be selected smaller than $\theta_0$, smaller than the
  solution to $4^f \theta^{2^{-f}} = \frac{1}{10}$, and small enough to have $2-4^f \theta^{2^{-f}}$
  between the diameter of the spherical cap centered at $p$ and $2$. All the parameters for our
  hypergraph are now chosen.

  Define a hypergraph $G = G(F,k,\epsilon,n)$ as follows.  Let $A$ be a set of size $|V(H_0(n_0, L,
  d, \linebreak[1] \theta))|$, let $C_1, \dots, C_{r-1}$ be sets of size $\frac{n}{r}$, and let $D$
  be a set of size $\frac{n}{r}$.  The vertex set of $G$ is the disjoint union $A \dot\cup C_1
  \dot\cup \cdots \dot\cup C_{r-1} \dot\cup D$.  To form the hyperedges of $G$, put a copy of
  $H_0(n_0,L,d,\theta)$ on $A$, add all hyperedges with one vertex in each $C_i$ and one vertex in
  $D$, and add the following hyperedges between $A$ and $C_i$.  Think of $C_i$ as a set of ``evenly
  spaced'' points on the unit sphere $\mathbb{S}^d$.  Make a hyperedge on $w = (p_1,\ldots,p_\ell)
  \in A$, $c_1 \in C_1, \ldots, c_{r-1} \in C_{r-1}$ if $\left\lVert p_i - c_j \right\rVert <
  \sqrt{2} - \beta$ for all $1 \leq i \leq \ell$ and all $1 \leq j \leq r-1$.
\end{defi}

To complete the first half of the proof of Theorem~\ref{mainthm}, we need to prove that $G$ has
large chromatic number, large minimum degree, and if $F$ is non-unifoliate $r$-partite then $F
\nsubseteq G$.

\begin{lemma} \label{lem:largechrom}
  Given $F$, $k$, and $\epsilon$, it is possible to select $n$ large enough so that
  $\chi(G(F,k,\epsilon,n)) \geq k$.
\end{lemma}

\begin{proof} 
If $n \geq n_0$, then Corollary~\ref{cor:Slargechrom} implies that $\chi(H_0) \geq k$.  Since $G[A] =
H_0$, $\chi(G) \geq k$.
\end{proof} 

\begin{lemma} \label{lem:constrmindeg}
  Given $F$, $k$, and $\epsilon$, it is possible to select $n$ large enough so that the minimum
  degree of $G = G(F,k,\epsilon,n)$ is at least $\left( \frac{(r-1)!}{r^{2r-2}} - \xi_r \epsilon
  \right) \binom{|V(G)|}{r-1}$, where $\xi_r$ is a constant depending only on $r$.
\end{lemma}

\begin{proof} 
Vertices in $C_i$ have degree at least $|C_1| \cdot \ldots \cdot |C_{i-1}| |C_{i+1}| \cdot \ldots
\cdot |D| = \left( \frac{n}{r} \right)^{r-1}$ and vertices in $D$ have degree $\prod |C_i| = \left(
\frac{n}{r} \right)^{r-1}$.  Consider some vertex $(p_1,\ldots,p_\ell)$ in $A$.  By
Lemma~\ref{lem:sphereprop} there are roughly $(1/2-\epsilon)|C_i|$ elements of $C_i$ within distance
$\sqrt{2} - \beta$ of each of $p_i$.  In the worst case, there are only $(2^{-\ell} - \xi_1\epsilon)
|C_i|$ elements of $C_i$ within $\sqrt{2} - \beta$ of all of $p_1, \ldots, p_\ell$, where $\xi_1$ is a
constant depending only on $\ell$ with $\xi_1 \epsilon$ much less than $2^{-\ell}$.  Thus the vertex
$(p_1, \ldots, p_\ell) $ will have degree at least $(2^{-\ell(r-1)} - \xi_2\epsilon) \prod \left| C_i
\right|$ where again $\xi_2$ is some constant depending only on $\ell$.  Therefore, if $n$ is
sufficiently large, the minimum degree of $G$ is at least
\begin{align} \label{eq:mindegree}
  \min \left\{ \left( \frac{n}{r} \right)^{r-1}, (2^{-\ell(r-1)} - \xi_2 \epsilon) \left( \frac{n}{r}
  \right)^{r-1} \right\} \geq \left(\frac{n}{r2^\ell}\right)^{r-1} - \epsilon \xi_3 n^{r-1},
\end{align}
where $\xi_3$ is a constant depending only on $r$.  Since $\ell = \left\lceil \log_2 r \right\rceil$,
$|V(G)| = n + |V(H_0)|$, and $|V(H_0)|$ is fixed, we can choose $n$ large enough so that
\begin{align*}
  \delta(G) \geq \frac{(r-1)!}{r^{2r-2}} \binom{|V(G)|}{r-1} - \epsilon \xi_4 n^{r-1},
\end{align*}
for some constant $\xi_4$ depending only on $r$.  The proof is now complete, but note that
the relative sizes of $C_i$ and $D$ could be optimized to obtain equality in the minimum in
\eqref{eq:mindegree} and thus a better bound on the minimum degree of $G$.
\end{proof} 

We now turn our attention to proving that if $F$ is non-unifoliate $r$-partite, then $G$ does not
contain a copy of $F$.  For this, we need a couple of helpful lemmas.  Throughout the rest of this
section, for $x,y \in \mathbb{S}^d$, $\rho(x,y)$ denotes the Euclidean distance between $x$ and $y$.

\begin{lemma} \label{lem:nearorfarpoints}
  Let $x$, $y$, and $z$ be points on $\mathbb{S}^d$ and let $0 <  a < \frac{1}{10}$.  If
  \begin{itemize}
    \setlength{\itemsep}{0pt}
    \setlength{\parskip}{0pt}
    \setlength{\parsep}{0pt}
    \item either $\rho(x,y) < a$ or $\rho(x,y) > 2 - a$,
    \item and either $\rho(y,z) < a$ or $\rho(y,z) > 2 - a$,
  \end{itemize}
  then either $\rho(x,z) < 4\sqrt{a}$ or $\rho(x,z) > 2 - 4\sqrt{a}$.
\end{lemma}

\begin{proof} 
There are four cases.  If $\rho(x,y) < a$ and $\rho(y,z) < a$, then the triangle inequality implies that
$\rho(x,z) < 2a$.  If $\rho(x,y) > 2 - a$ and $\rho(y,z) < a$, then let $x'$ be the point antipodal to $x$ on
the sphere.  Since $x,x',y$ forms a right triangle (with right angle at $y$),
\begin{align*}
  \rho^2(x',y) = \rho^2(x,x') - \rho^2(x,y) < 4 - (2-a)^2 \leq 4a,
\end{align*}
so by the triangle inequality, $\rho(x',z) \leq \rho(x',y) + \rho(y,z) \leq 2\sqrt{a} + a \leq 3\sqrt{a}$.
Now using that $x,x',z$ forms a right triangle (with right angle at $z$),
\begin{align*}
  \rho^2(x,z) = \rho^2(x,x') - \rho^2(x',z) \geq 4 - 9a > 4 - 16\sqrt{a} + 16a = (2-4\sqrt{a})^2.
\end{align*}
(The last inequality used that $-9a > -16\sqrt{a} + 16a$ since $a < \frac{1}{10}$.)  This
finishes the case that $\rho(x,y) > 2-a$ and $\rho(y,z) < a$ and also finishes the case $\rho(x,y) < a$ and
$\rho(y,z)> 2-a$ by symmetry.  The last case is when $\rho(x,y) > 2 - a$ and $\rho(y,z) > 2-a$.  Let $y'$ be
the point antipodal to $y$ on the sphere.  Then
\begin{align*}
  \rho^2(x,y') &= \rho^2(y,y') - \rho^2(x,y) < 4 - (2-a)^2 \leq 4a, \\
  \rho^2(z,y') &= \rho^2(y,y') - \rho^2(z,y) < 4 - (2-a)^2 \leq 4a, \\
  \rho(x,z) &\leq \rho(x,y') + \rho(y',z) < 4\sqrt{a}.
\end{align*}
\end{proof} 

\begin{lemma} \label{lem:pointsclose}
  Let $x = (x_1,\dots,x_\ell)$ and $y = (y_1,\dots,y_\ell)$ be two distinct vertices of $H_0 =
  H_0(n_0,L,d,\theta)$ contained together in a hyperedge of $H_0$.  Then for every $1 \leq i \leq
  \ell$, either $\rho(x_i, y_i) > 2 - 4\sqrt{\theta}$ or $\rho(x_i, y_i) < 4\sqrt{\theta}$.
\end{lemma}

\begin{proof} 
Let $E$ be the edge containing $x$ and $y$ and let $\bar{z}^1, \dots, \bar{z}^r$ be an ordering of
the vertices of $E$ such that if $ab \in E(B_i)$ then $\rho(z^a_i, z^b_i) > 2 - \theta$ (recall that
the vertex $\bar{z}^j$ consists of an $\ell$-tuple $(z^j_1,\dots,z^j_\ell)$ with each $z^j_i$ a
point on the unit sphere).  Let $1 \leq a \neq b \leq r$ such that $x = \bar{z}^a$ and $y =
\bar{z}^b$.  If $ab \in E(B_i)$ then $\rho(x_i, y_i) > 2 - \theta$.  If $ab \notin E(B_i)$, they
appear in the same part of the bipartite graph $B_i$ so let $c$ be a common neighbor of $a$ and $b$
in $B_i$.  Now since $ac, cb \in E(B_i)$, we have that $\rho(x_i, z^c_i) > 2 - \theta$ and
$\rho(y_i, z^c_i) > 2 - \theta$ so that Lemma~\ref{lem:nearorfarpoints} implies that $\rho(x_i, y_i)
< 4\sqrt{\theta}$ or $\rho(x_i, y_i) > 2 - 4\sqrt{\theta}$ (in fact the proof of
Lemma~\ref{lem:nearorfarpoints} implies that $\rho(x_i,y_i) < 4\sqrt{\theta}$).
\end{proof} 

\begin{lemma} \label{lem:closeorfar}
  Let $x = (x_1,\dots,x_\ell)$ and $y = (y_1,\dots,y_\ell)$ be two distinct vertices of $H_0 =
  H_0(n_0,L,d,\theta)$ which appear in the same component of $H_0$ and are at distance at most $f =
  |E(F)|$.  That is, there exist hyperedges $E_1, \dots, E_q$ of $H_0$ where $q \leq f$, $x \in
  E_1$, $y \in E_q$, and $E_i \cap E_{i+1} \neq \emptyset$.  Then for every $1 \leq i \leq \ell$,
  either $\rho(x_i, y_i) > 2 - 4^{f}\theta^{2^{-f}}$ or $\rho(x_i, y_i) < 4^{f}\theta^{2^{-f}}$.
\end{lemma}

\begin{proof} 
Consider a path $E_1, \dots, E_q$ from $x$ to $y$ with $q \leq f$ and fix $1 \leq i \leq \ell$.
For $1 \leq j \leq q - 1$, let $e^j \in E_j \cap E_{j+1}$.  Recall that $x = (x_1, \dots, x_\ell)$,
$y = (y_1,\dots,y_\ell)$, and $e^j = (e^j_1,\dots,e^j_\ell)$ where $x_i$, $y_i$, and $e^j_i$ are
points on the sphere.  We will prove by induction on $j$ that either $\rho(x_i, e^j_i) < 4^{j}
\theta^{2^{-j}}$ or $\rho(x_i, e^j_i) > 2 - 4^{j} \theta^{2^{-j}}$.  The base case of $j = 1$ is
just Lemma~\ref{lem:pointsclose}, since $x$ and $e^1$ are contained together in a hyperedge.  For
the inductive step, the fact that $e^j$ and $e^{j+1}$ are contained together in a hyperedge,
Lemma~\ref{lem:pointsclose}, and induction imply that
\begin{align*}
  &\text{either} \quad \rho(x_i, e^j_i) < 4^{j} \theta^{2^{-j}} \quad \text{or} \quad
                   \rho(x_i, e^j_i) > 2 - 4^{j} \theta^{2^{-j}}, \\
  &\text{either} \quad \rho(e^j_i, e^{j+1}_i) < 4\sqrt{\theta} \quad \text{or} \quad
                   \rho(e^j_i, e^{j+1}_i) > 2 - 4\sqrt{\theta}.
\end{align*}
Since $4^{j} \theta^{2^{-j}} \geq 4\sqrt{\theta}$, Lemma~\ref{lem:nearorfarpoints} shows that (note
we selected $\theta$ so that $4^f \theta^{2^{-f}} < \frac{1}{10}$)
\begin{align*}
  &\text{either} \quad \rho(x_i, e^{j+1}_i) < 4 \sqrt{4^{j} \theta^{2^{-j}} } \quad \text{or} \quad
                   \rho(x_i, e^{j+1}_i) > 2 - 4 \sqrt{4^{j} \theta^{2^{-j}} }.
\end{align*}
Since $4 \sqrt{4^{j} \theta^{2^{-j}} } \leq 4^{j+1} \theta^{2^{-j-1}}$, the proof of the inductive
step is complete.

Therefore either $\rho(x_i, e^{q-1}_i) < 4^{q-1} \theta^{2^{-q+1}}$ or $\rho(x_i, e^{q-1}_i) > 2 -
4^{q-1} \theta^{2^{-q+1}}$.  But since $e^{q-1}$ and $y$ are contained together in the hyperedge
$E_q$, one last application of Lemmas~\ref{lem:nearorfarpoints} and~\ref{lem:pointsclose} show that
either $\rho(x_i, y_i) < 4^q \theta^{2^{-q}}$ or $\rho(x_i, y_i) > 2 - 4^q \theta^{2^{-q}}$.  Since
$q \leq f$, the proof is complete.
\end{proof} 

\begin{lemma} \label{lem:constrnounifoliate}
  If $F$ is not unifoliate $r$-partite, then $G$ does not contain a copy of $F$.
\end{lemma}

\begin{proof} 
Assume $G$ contains a copy of $F$; we will prove that $F$ is unifoliate $r$-partite, a
contradiction.  Any copy of $F$ must use vertices of $A$ since otherwise it would be $r$-partite.
It cannot be completely contained in $A$ since by construction any $L$ vertices in $A$ induce a
linear hyperforest and we set $L = |V(F)|$.  Let $F'$ be the subhypergraph of $F$ formed by
restricting $F$ to $A \cup C_1 \cup \dots \cup C_{r-1}$.  We first show that $F'$ is unifoliate
$r$-partite and then secondly show that this implies that $F$ is unifoliate $r$-partite.

\medskip

\noindent\textbf{Claim 1:} $F'$ is unifoliate $r$-partite.

\begin{proof} 
Let $X = V(F') \cap A$ so $G[X]$ is a linear hyperforest.  We claim that $V_1 = X$, $V_2 = V(F')
\cap C_1, \dots, V_r = V(F') \cap C_{r-1}$ is a partition witnessing that $F'$ is unifoliate
$r$-partite.  First, $F'[X]$ is a linear hypertree since $G[X]$ is a linear hypertree.  Also, all
other edges cross the partition since edges in $G$ not completely contained in $A$ and not using
vertices of $D$ use one vertex of $A$ and one vertex from each $C_i$.  Now consider a cycle $C$
using exactly one edge $E$ of $G[X]$ and cross-edges which intersect $A$ only in vertices in the
same component of the hyperforest as $E$.  We will show that such a cycle does not exist by deriving
a contradiction.

Let $E_1, \dots, E_m$ be the edges of the cycle $C$ labeled so that $E_1$ is the edge of $C$ in
$G[X]$ and $E_i \cap E_{i+1} \neq \emptyset$.  Let $x \in E_1 \cap E_2$ and $y \in E_1 \cap E_m$ so
that $x \neq y$.  Since the bipartite graphs $B_1, \dots, B_\ell$ cover $K_{[r]}$, there is some $1
\leq i \leq \ell$ such that $\rho(x_i, y_i) > 2 - \theta$.  Fix such an $i$ for the remainder of
this proof.  For $2 \leq s \leq m$, let $e^s \in E_s \cap A$ (there is a unique such $e^s$ since
these are all cross-edges).  Note that $x = e^2$ and $y = e^m$.  Finally, let $f = |E(F)|$.

We now claim that for all $2 \leq s \leq m$, $\rho(x_i, e^s_i) < s 4^{f}\theta^{2^{-f}}$.  Since
$e^m = y$, this will contradict that  $\rho(x_i, y_i) > 2 - \theta$.  We prove that $\rho(x_i,e^s_i)
< s4^{f}\theta^{2^{-f}}$ by induction.  First, since $C$ uses vertices in $H_0[X]$ within the
same component as $E_1$, the vertices $e^s$ are all within the same component of $H_0$ so
Lemma~\ref{lem:closeorfar} implies that either their $i$th coordinates (pairwise) are within
distance $4^f\theta^{2^{-f}}$ of each other or have distance at least $2-4^f\theta^{2^{-f}}$.  In
particular, either $\rho(e^s_i, e^{s+1}_i) < 4^f\theta^{2^{-f}}$ or $\rho(e^s_i, e^{s+1}_i) > 2 -
4^f\theta^{2^{-f}}$ and we claim that the latter is impossible.  First, if $e^s = e^{s+1}$ then
obviously $\rho(e^s_i, e^{s+1}_i) = 0 < 4^f \theta^{2^{-f}}$.  For $e^s \neq e^{s+1}$, since $E_s
\cap E_{s+1} \neq \emptyset$, there exists a vertex $z \in E_s \cap E_{s+1} \cap (C_1 \cup \dots
\cup C_{r-1})$.  Since $z$ is contained together with $e^s$ in the cross-hyperedge $E_s$, by the
definition of cross-edges of $G$ we have that $\rho(e^s_i, z_i) < \sqrt{2} - \beta$.  Similarly,
$\rho(e^{s+1}_i, z_i) < \sqrt{2}-\beta$.  By \eqref{eq:defofthetadiam} (with $p = z_i$),
$\rho(e^{s+1}_i, e^s_i) < 2 - 4^f\theta^{2^{-f}}$ since both are within distance $\sqrt{2} - \beta$
of $z_i$.  By Lemma~\ref{lem:closeorfar}, this implies that $\rho(e^s_i, e^{s+1}_i) < 4^f
\theta^{2^{-f}}$.  By induction and the triangle inequality, $\rho(e^2_i, e^s_i) <
s4^f\theta^{2^{-f}}$.  Since $x = e^2$ and $y = e^m$, $\rho(x_i, y_i) < m4^f\theta^{2^{-f}}$, a
contradiction of the coordinate choice $i$.  This contradiction proves that no such cycle can exist
so that $V_1, \dots, V_r$ witnesses that $F'$ is unifoliate $r$-partite.
\end{proof} 

\medskip

\noindent\textbf{Claim 2:} $F$ is unifoliate $r$-partite.

\begin{proof} 
We will show that we can extend $V_1, \dots, V_t$ to be a witness to the fact that $F$ is unifoliate
$r$-partite, contradicting the choice of $F$.  Indeed, let $W_1 = V_1 \cup (V(F) \cap D), W_2 = V_2,
\dots, W_t = V_t$ be a partition of $V(F)$.  That is, add all vertices which appear in $V(F) \cap D$
to $V_1$.  Note that edges of $F$ are either contained in $W_1$ or use one vertex from each $W_i$,
since the edges of $F$ touching $D$ use one vertex from $D$ and one vertex from each $C_i$ and
$W_{i+1} \subseteq C_i$.  Also, $F[W_1]$ is still a linear hyperforest since no edges of $F - F'$
were added inside $W_1$.  Finally, consider a cycle $C$.  If $C$ uses no edges intersecting $D$ then
it is a cycle in $F'$ so it is good.  If $C$ uses an edge $E$ of $F - F'$, then let $\{x\} = E \cap
W_1$ and notice that $x$ is an isolated vertex in $F[W_1]$.  Since $C$ must use at least one more
edge and $x$ is isolated, the cycle either uses another vertex of $W_1$ so uses vertices from at
least two components of $F[W_1]$, or $C$ uses only the vertex $x$ from $W_1$ and so uses zero edges
of $F[W_1]$.  Thus $W_1, \dots, W_t$ witnesses that $F$ is unifoliate $r$-partite.
\end{proof} 
Claim 2 contradicts the assumption that $F$ is not unifoliate $r$-partite, completing the proof.
\end{proof} 

We shall now prove half of Theorem~\ref{mainthm}, i.e.\ that if $F$ is a non-unifoliate $r$-partite
hypergraph, then the family of $F$-free hypergraphs has chromatic threshold at least
$(r-1)!r^{-2r+2}$.  We must prove that the infimum of the values $c$ such that the family of
$F$-free hypergraphs $H$ with minimum degree at least $c\binom{|V(H)|}{r-1}$ has bounded chromatic
number.  So consider $c < \frac{(r-1)!}{r^{2r-2}}$ and assume the family
\begin{align*}
  \mathcal{H} = \left\{ H : F \nsubseteq H, \delta(H) \geq c \binom{|V(H)|}{r-1} \right\}
\end{align*}
has bounded chromatic number.  That is, there exists some integer $k$ so that $\chi(H) < k$ for
every $H \in \mathcal{H}$.  Now define $\epsilon > 0$ so that
\begin{align*}
  c < \frac{(r-1)!}{r^{2r-2}} - \xi_r \epsilon
\end{align*}
where $\xi_r$ is the constant depending only on $r$ from Lemma~\ref{lem:constrmindeg}.  Now that we
have chosen $F$, $k$, and $\epsilon$, by Lemmas~\ref{lem:largechrom},~\ref{lem:constrmindeg},
and~\ref{lem:constrnounifoliate}, we can select $n$ large enough so that $G = G(F,k,\epsilon,n)$ is
an element of $\mathcal{H}$ with chromatic number at least $k$, a contradiction.

%% file: ctz-upper.tex
\section{Strong Unifoliate $r$-partite hypergraphs} 
\label{sec:upperbound}

In this section, we prove that if $G$ is an $r$-uniform, strong unifoliate $r$-partite hypergraph, then the
family of $G$-free hypergraphs has chromatic threshold zero.

\begin{defi}
  Let $d$ be an integer, $H$ an $r$-uniform hypergraph, and $v \in V(H)$.  The hypergraph $H$ is
  \emph{$(d,v)$-bounded} if there is a vertex set $A_v \subseteq V(H)$ with $\left| A_v \right| \leq
  d$ such that every edge of $H$ containing $v$ intersects $A_v$.  A hypergraph $H$ is
  \emph{$d$-degenerate} if there exists an ordering $v_1, \ldots, v_n$ of the vertices such that for
  every $1 \leq i \leq n$, the hypergraph $H[v_1,\ldots,v_i]$ is $(d,v_i)$-bounded.
\end{defi}

\begin{lemma} \label{lem:degencolor}
  Let $H$ be a $d$-degenerate hypergraph.  Then $\chi(H) \leq d+1$.
\end{lemma}

\begin{proof} 
Use a greedy coloring of the vertices in the order $v_1, \dots, v_n$ given by the definition of
$d$-degenerate.
\end{proof} 

\begin{lemma} \label{lem:degen}
  Let $G$ be an $r$-uniform linear hyperforest on $d$ vertices and let $H$ be an $r$-uniform,
  non-$d$-degenerate hypergraph.  Then $H$ contains a copy of $G$.
\end{lemma}

\begin{proof} 
Observe that there is some induced subhypergraph $H'$ of $H$ such that for every $v \in V(H')$ and
every subset $A \subseteq V(H')$ with $|A| \leq d$, there is a hyperedge of $H'$ containing $v$ and
missing $A$.  Indeed, delete vertices of $H$ one by one until the condition is satisfied.  If all
vertices were deleted then $H$ is $d$-degenerate, a contradiction.

Now embed the edges of $G$ greedily into $H'$.  Since $G$ is a linear hyperforest, its edges can be
ordered $E_1, \dots, E_m$ so that $E_i$ uses at most one vertex from $E_1 \cup \dots \cup E_{i-1}$.
To embed $E_i$, let $A$ be the set of previously embedded vertices and $x$ the previously embedded
vertex which $E_i$ must extend if $E_i \cap (E_1 \cup \dots \cup E_{i-1}) \neq \emptyset$ and
otherwise let $x$ be any vertex of $H'$ outside $A$.  Since $|A| \leq d$, the definition of $H'$
guarantees the existence of a hyperedge missing $A-x$ to which $E_i$ can be embedded.
\end{proof} 

\newcommand{\rb}{r_B}
\newcommand{\rg}{r_{\gamma}}
\begin{defi}
  The following definitions are from~\cite{cr-balogh11}.  A \emph{fiber bundle} is a tuple $(B,
  \gamma, F)$ such that $B$ is a hypergraph, $F$ is a finite set, and $\gamma : V(B) \rightarrow
  2^{2^F}$.  That is, $\gamma$ maps vertices of $B$ to collections of subsets of $F$, which we can
  consider as hypergraphs on vertex set $F$.  The hypergraph $B$ is called the \emph{base
  hypergraph} of the bundle and $F$ is the \emph{fiber} of the bundle.  For a vertex $b \in V(B)$,
  the hypergraph $\gamma(b)$ is called the \emph{fiber over $b$}.  A fiber bundle $(B,\gamma,F)$ is
  \emph{$(\rb,\rg)$-uniform} if $B$ is an $\rb$-uniform hypergraph and $\gamma(b)$ is an
  $\rg$-uniform hypergraph for each $b \in V(B)$.  Given $X \subseteq V(B)$, the \emph{section of
  $X$} is the hypergraph with vertex set $F$ and edges $\cap_{x \in X} \gamma(x)$.  In other words,
  the section of $X$ is the collection of subsets of $F$ that appear in the fiber over $x$ for every
  $x \in X$. For a hypergraph $H$, define $\dim_H(B, \gamma, F)$ to be the maximum integer $d$ such
  that there exist $d$ pairwise disjoint edges $E_1, \ldots, E_d$ of $B$ (i.e.\ a matching) such
  that for every $x_1 \in E_1, \ldots, x_d \in E_d$, the section of $\left\{ x_1, \ldots, x_d
  \right\}$ contains a copy of $H$.
\end{defi}

Balogh, Butterfield, Hu, Lenz, and Mubayi~\cite{cr-balogh11} proved the following theorem about fiber bundles.

\begin{thm} \label{thm:bundledim}
  Let $\rb \geq 2$, $\rg \geq 1$, $d \in \mathbb{Z}^{+}$, $0 < \epsilon < 1$, and $K$ be an
  $\rg$-uniform hypergraph with zero Tur\'{a}n density.  Then there exists constants $C_1 =
  C_1(\rb,\rg,d,\epsilon,K)$ and $C_2 = C_2(\rb,\rg,d,\epsilon,K)$ such that the following holds.
  Let $(B,\gamma,F)$ be any $(\rb,\rg)$-uniform fiber bundle where $\dim_K(B,\gamma,F) < d$ and for
  all $b \in V(B)$, $\left| \gamma(b) \right|  \geq \epsilon \binom{\left| F \right|}{\rg}$. If
  $\left| F \right| \geq C_1$, then $\chi(B) \leq C_2$.
\end{thm}

We will apply Theorem~\ref{thm:bundledim} to the following fiber bundle.

\begin{defi}
  Given $r$-uniform hypergraphs $T$ and $H$, define the \emph{$T$-bundle of $H$} as the following
  fiber bundle.  Let $B$ be the hypergraph with vertex set $V(H)$, where a set $X \subseteq V(B)$ is
  a hyperedge of $B$ if $|X| = |V(T)|$ and $H[X]$ contains a (not necessarily induced) copy of $T$.
  Let $F = V(H)$.  Define $\gamma : V(B) \rightarrow 2^{2^F}$ as the map which sends $b \in V(B)$ to
  $\{A \subseteq F : A \cup \{b\} \in E(H)\}$.  The \emph{$T$-bundle of $H$} is the fiber bundle
  $(B,\gamma,F)$.
\end{defi}

A simple corollary of Lemmas~\ref{lem:degencolor} and~\ref{lem:degen} is the following.

\begin{cor} \label{lem:degenbundle}
  Let $T$ be an $r$-uniform linear hyperforest, let $H$ be an $r$-uniform hypergraph, and let
  $(B,\gamma,F)$ be the $T$-bundle of $H$.  Then $\chi(H) \leq (|V(T)|+1) \chi(B)$.
\end{cor}

\begin{proof} 
Let $C \subseteq V(B)$ be a color class in a proper $\chi(B)$-coloring of $B$.  If $H[C]$ is
non-$|V(T)|$-degenerate, then by Lemma~\ref{lem:degen} $H[C]$ contains a copy of $T$.  This copy of
$T$ becomes an edge of $B$ contained in $C$ contradicting that the coloring of $B$ is proper.  Thus
$H[C]$ is $|V(T)|$-degenerate so Lemma~\ref{lem:degencolor} implies that $\chi(H[C]) \leq |V(T)| + 1$.
Combining these two colorings produces a proper $(|V(T)|+1)\chi(B)$-coloring of $H$.
\end{proof} 

We are now ready to complete the proof of Theorem~\ref{mainthm}, i.e.\ prove that if $G$ is an
$r$-uniform, strong unifoliate $r$-partite hypergraph, then the family of $r$-uniform, $G$-free
hypergraphs has chromatic threshold zero.  Let $G$ be an $r$-uniform, strong unifoliate $r$-partite
hypergraph with $m$ vertices. Let $H$ be an $r$-uniform, $G$-free hypergraph with $\delta(H) \geq
\epsilon \binom{|V(H)|}{r-1}$.  We need to show that the chromatic number of $H$ is bounded by a
constant depending only on $\epsilon$ and $G$.  Let $V_1, \dots, V_r$ be a vertex partition of
$V(G)$ guaranteed by the definition of strong unifoliate $r$-partite and let $t$ be the number of
components in $G[V_1]$. Let $(B,\gamma,F)$ be the $G[V_1]$-bundle of $H$ and let $K$ be the complete
$(r-1)$-uniform, $(r-1)$-partite hypergraph with $(rm)^m$ vertices in each part.  Since $\delta(H)
\geq \epsilon \binom{|V(H)|}{r-1}$, for every $x \in V(H) = V(B)$, we have $|\gamma(x)| = d(x) \geq
\delta(H) \geq \epsilon \binom{|V(H)|}{r-1}$.  Since $F = V(H)$, we have that for every $x \in
V(B)$, $|\gamma(x)| \geq \epsilon \binom{|F|}{r-1}$.

First, assume that $\dim_K(B,\gamma,F) < t$. Notice that $(B,\gamma,F)$ is $(|V_1|, r-1)$-uniform
by definition and $t$ depends on $G$, so Theorem~\ref{thm:bundledim} implies that if $|F| \geq C_1$
then $\chi(B) \leq C_2$, where $C_1$ and $C_2$ are constants depending only on $\epsilon$ and $G$.
Since $F = V(H) = V(B)$, this implies that $\chi(B) \leq \max\{C_1,C_2\}$ so $\chi(B)$ is bounded by
a constant depending only on $\epsilon$ and $G$.  Now Corollary~\ref{lem:degenbundle} shows that the
chromatic number of $H$ is bounded by a constant depending only on $\epsilon$ and $G$, exactly what
we would like to prove.

Therefore, $\dim_K(B,\gamma,F) \geq t$. Let $T_1, \ldots, T_t$ be the components of $G[V_1]$.  We
now show how to find a copy of $G$ in $H$.  Since $\dim_K(B,\gamma,F) \geq t$, there exists edges
$E_1, \ldots, E_t$ of $B$ with large sections. Recall that by the definition of $B$, for every $i$
$H[E_i]$ contains a copy of $G[V_1]$ so we can find a copy of $T_i$ in $H[E_i]$ for each $i$.  Pick
any $x_1 \in V(T_1), \ldots, x_t \in V(T_t)$. The definition of $\dim_K(B,\gamma,F)$ implies that
the section of $x_1, \ldots, x_t$ contains a copy of $K$, the complete $(r-1)$-uniform,
$(r-1)$-partite hypergraph with $(rm)^m$ vertices in each class.  These copies of $K$ are used to
embed the rest of $G$ in $H$ as follows.

Let $G'$ be the $(r-1)$-uniform hypergraph with vertex set $V_2 \cup \dots \cup V_r$ where
$\{v_2,\dots,v_r\}$ is a hyperedge of $G'$ if there exists a cross-hyperedge of $G$ containing
$\{v_2,\dots,v_r\}$.  Let $C_1, \dots, C_c$ be the components $G'$ and define 
\begin{align*}
  N_{V_1}(C_i) = \{ x \in V_1 : \exists \{v_2,\dots,v_r\} \in E(C_i) \, \text{with } \{x, v_2,
  \dots, v_r\} \in E(G)\}.
\end{align*}
In other words, $N_{V_1}(C_i)$ is the ``neighborhood'' of $C_i$ in $G$, the collection of vertices
of $V_1$ which are contained in a hyperedge of $G$ together with some $(r-1)$-edge of $C_i$.  Note
that since $G$ is strong unifoliate $r$-partite, each $C_i$ is an $(r-1)$-partite, $(r-1)$-uniform
hypergraph. Also since $G$ is strong unifoliate $r$-partite, for each $C_i$ and each hypertree $T_j$, there
is at most one vertex of $T_j$ in $N_{V_1}(C_i)$.  Let $x_{i,j}$ be such a vertex if it exists and
otherwise define $x_{i,j}$ to be any vertex in $T_j$.  We can now find a copy of $G$ in $H$ by
embedding each $C_i$ into the copy of $K$ contained in the section of $x_{i,1}, \dots, x_{i,t}$.
This forms a copy of $G$ in $H$ since $V(C_1) \dot\cup \cdots \dot\cup V(C_c)$ is a partition of
$V_2 \cup \dots \cup V_r$ and each $C_i$ is embedded into exactly one of the sections of vertices
from the hypertrees in $V_1$.